\documentclass[letterpaper, 10 pt, conference]{ieeeconf}

\IEEEoverridecommandlockouts               
                                                    
\usepackage[pdftex]{graphicx}
\graphicspath{{../pdf/}{../jpeg/}}
\DeclareGraphicsExtensions{.pdf,.jpeg,.png}
\usepackage{amssymb,amsmath,url}
\usepackage{bm} 
\usepackage{multirow}
\usepackage{booktabs}	
\usepackage{epstopdf}
\usepackage{color}
\usepackage{bigfoot}		
\usepackage{algpseudocode}
\usepackage{pifont}
\usepackage{tikz}
\usetikzlibrary{intersections}
\usepackage{color}
\usepackage{cite}
\usepackage{lipsum}
\usepackage{algorithm} 
\usepackage[pdftex]{graphicx}
\usepackage{setspace}
\usepackage{makecell}
\usepackage{comment}
\usepackage{physics}
\usepackage{enumerate}

\definecolor{barblue}{RGB}{153,204,254}
\definecolor{groupblue}{RGB}{51,102,254}
\definecolor{linkred}{RGB}{165,0,33} 
\definecolor{mycolor1}{rgb}{0.00000,0.44700,0.74100}%
\definecolor{mycolor2}{rgb}{0.85000,0.32500,0.09800}%
\definecolor{mycolor3}{rgb}{0.92900,0.69400,0.12500}%

\algrenewcommand\algorithmicrequire{\textbf{Input:}}
\algrenewcommand\algorithmicensure{\textbf{Output:}}
\algrenewcommand\algorithmicforall{\textbf{For}}

\newtheorem{proposition}{Proposition}

 \allowdisplaybreaks

\newtheorem{remark}{Remark}

\newcommand{\R}{\mathbb{R}}

\title{\LARGE \bf 
A convex relaxation approach for the optimized pulse pattern problem \\
}
\author{Lukas Wachter, Orcun Karaca, Georgios Darivianakis, and Themistoklis Charalambous
\thanks{L. Wachter and T. Charalambous are with Aalto University, Espoo, Finland. e-mails: {\tt wachter.lukas@t-online.de, themistoklis.charalambous@aalto.fi}. \newline \indent O. Karaca is with the Automatic Control Laboratory, ETH Z\"{u}rich, Switzerland. e-mail: {\tt okaraca@ethz.ch}. \newline \indent G. Darivianakis is with ABB Corporate Research Center, Baden, Switzerland. email:  {\tt georgios.darivianakis@ch.abb.com}.}%
}
\allowdisplaybreaks
\begin{document}
\maketitle
\thispagestyle{empty}
\pagestyle{empty}
\begin{abstract}\noindent Optimized Pulse Patterns (OPPs) are gaining increasing popularity in the power electronics community over the well-studied pulse width modulation due to their inherent ability to provide the switching instances that optimize current harmonic distortions. In particular, the OPP problem minimizes current harmonic distortions under a cardinality constraint on the number of switching instances per fundamental wave period. The OPP problem is, however, non-convex involving both polynomials and trigonometric functions. In the existing literature, the OPP problem is solved using off-the-shelf solvers with local convergence guarantees. To obtain guarantees of global optimality, we employ and extend techniques from polynomial optimization literature and provide a solution with a global convergence guarantee. Specifically, we propose a polynomial approximation to the OPP problem to then utilize well-studied globally convergent convex relaxation hierarchies, namely, semi-definite programming and relative entropy relaxations. The resulting hierarchy is proven to converge to the global optimal solution. Our method exhibits a strong performance for OPP problems up to 50 switching instances per quarter wave. 
\end{abstract} 
\vspace{.2cm}

\section{Introduction}
In recent years, combustion engines have increasingly been either augmented or replaced by electric drives. This shift has resulted in a boom in demand for power electronics mainly driven by their higher efficiency, higher reliability and comparatively low prices~\cite{orlowska2016, mecrow2008}. To further improve the efficiency of these devices, one needs to study the efficiency for power conversion. Widely-used power converters are those that switch on and off a constant voltage at certain times to generate an approximated sinusoidal pattern of a desired frequency and magnitude, which inevitably contains harmonics due to the discontinuities at the switching instances. In medium to high voltage applications, these switches are realized with power semiconductors that also have significant switching losses~\cite{rashid2001}. Thus, there is a fundamental trade-off between the switching frequency/losses and the resulting harmonic distortion. 

Optimized Pulse Pattern (OPP) is a pulse width modulation scheme that minimizes the current distortion for a given switching frequency~\cite{boller2011}.
A low switching frequency requirement is crucial in high-power applications where the switching losses in the semiconductor devices is the main bulk of losses in the converter. The OPP switching patterns are computed offline and stored in a look-up-table for a desired fundamental frequency, magnitude, and number of switching instances per wave period~\cite{geyer2012}. This is necessary due to the complexity of solving such optimization problems online. Specifically, both the objective function and the constraints of OPP problems are non-convex. Thus, traditional solvers (for instance, those utilizing interior point methods) require a large amount of initial starting values without providing any certificate on global optimality. 

On the other hand, there are well-known methods in non-convex optimization literature that yield a global optimum if the non-convexity is originating from a polynomial objective function and polynomial constraints~\cite{lasserre2001, parrilo2000, chandrasekaran2016, ahmadi2017,parrilo2004, lasserrebook2015,summers2013approximate,prajna2002introducing}. 
One such method is given by sum of squares (SOS)/moment hierarchy. Developed by the works in~\cite{parrilo2000, lasserre2001}, the method utilizes a hierarchy of increasingly larger semi-definite programs as convex relaxations of the original non-convex problem, producing a convergent series of lower bounds to the optimal value while also proposing solution recovery methods. 
The convergence is typically achieved at a finite level of the hierarchy. While computationally tractable, the SOS approach still requires a considerable amount of resources for large degrees or dimensions. As an alternative, \cite{ahmadi2017} proposes a hierarchy based on linear programming and second-order cone programming to produce lower bounds more efficiently by relaxing the convergence guarantee~\cite{josz2017}. Other advancements towards tractability focus on symmetry reductions and exploiting sparsity~\cite{yang2020, lasserre2015,gatermann2004symmetry,waki2006}. 
Recently,~\cite{chandrasekaran2016} proposed a convex relaxation hierarchy for polynomial optimization problems by utilizing the relative entropy (exponential) cone instead of the positive semidefinite cone, see also the developments in~\cite{murray2018, murray2019,dressler2017positivstellensatz,karaca2017,dressler2018dual,katthan2019unified,murray2020x}. This hierarchy was originally developed for signomial optimization, which is polynomial optimization over the positive orthant. These methods can also be extended for polynomial optimization problems~\cite{murray2019}.  The advantages of this hierarchy compared to the SOS is that it can handle problems with high degree and dimension if the number of unique monomials are limited.
It is even possible to combine these two hierarchies to generate alternative convex relaxations~\cite{karaca2017}.

The contribution of this paper is to employ these convex relaxation methods to solve the OPP problem and to obtain a global solution without relying on local methods. We formulate an approximation to the non-convex, non-polynomial optimization problem by using only polynomials. This reformulation relies on polynomial approximations to the equality constraints involving trigonometric functions. We then show that these approximations form a convergent hierarchy, which is not guaranteed to exhibit monotonic improvements. We illustrate how this method provides us with sufficiently precise results at a certain level of the approximation when the allowed computation time is limited. To the best of our knowledge, this is the first work to utilize convex relaxation hierarchies to obtain a globally optimal solution to the OPP problem. 

The paper is structured as follows. Section~\ref{sec:polyopt} introduces tools from polynomial optimization. Section~\ref{sec:pwm} introduces optimized pulse patterns. In Section~\ref{sec:main}, we present the problem formulation and then formulate a polynomial approximation to the problem. We show that our approximation produces a convergent hierarchy. We also perform a numerical study to evaluate the performance of the resulting hierarchy. Section~\ref{sec:conc} concludes the paper with a discussion on the findings and an outlook for further research.

\textit{Notation:} Bold face denotes a vector, $\bm{x}\in \mathbb{R}^d$ and $x_i\in \mathbb{R}$ denotes the $i$-th element of that vector. A generic multivariate polynomial is denoted as $f(\bm{x}) = \sum_{j=1}^l c_j \bm{x}^{\bm{\beta}^{(j)}}$ with real coefficients $c_j \in \mathbb{R}$ for all $j=1,\dots,l$. $\bm{x}^{\bm{\beta}} = x_1^{\beta_1}\cdots x_d^{\beta_d}$ where $\bm{\beta}$ is the vector of exponents. $\mathrm{deg}(f(\bm{x}) = \mathrm{max}_j \bm{1}^\top \bm{\beta}^{(j)}$ is the total degree of the polynomial. The list of exponent vectors $\{\bm{\beta}^{(j)}\}$ is such that $\bm{\beta}^{(j)}\in \mathbb{Z}^d$ for all $j=1,\dots,l$. The real valued polynomials are denoted as $f(\bm{x}) \in \mathbb{R}[\bm{x}]$ where $\mathbb{R}[\bm{x}]$ is the the polynomial ring in $\bm{x}\in\mathbb{R}^d$ with real coefficients. The ceiling operator is denoted as $\lceil x \rceil$. To denote that $A$ is a positive semidefinite matrix, we write $A\succeq 0$.

\section{Preliminaries on Polynomial Optimization Methods}\label{sec:polyopt}

Polynomial optimization is a class of optimization problems where the objective function and the constraints are of the form of polynomials. Polynomial optimization in its general form is an NP-hard problem~\cite{parrilo2004, lasserrebook2015}. However, solving a non-convex polynomial optimization problem globally is achieved by solving a series of convex relaxations. 

Assume a generic multivariate polynomial $f(\bm{x}) = \sum_{j=1}^l c_j \bm{x}^{\bm{\beta}^{(j)}}$ with the exponent vector $\{\bm{\beta}^{(j)}\}_{j=1}^l \subset \mathbb{Z}_+^d$ with $\bm{\beta}^{(1)} = \bm{0}$ and the coefficients $c_j\in\mathbb{R}$. 
To find its minimum subject to some polynomial inequality constraints, one can utilize the hypograph formulation:
\begin{align}
\begin{split}
f^* = \displaystyle\underset{\gamma}{\max} &\quad \gamma\\
\mathrm{s.t.} & \quad f(\bm{x}) - \gamma \geq 0, \quad \forall \bm{x}\in K ,
\end{split}\label{eq:hypograph}
\end{align}
where $K = \{\bm{x}\in \mathbb{R}^{d}\,|\, g_i(\bm{x}) \geq 0, \ \forall  i\}$ defines a semi-algebraic set.
In other words, the problem boils down to finding the largest $\gamma$ such that the polynomial $f(\bm{x})-\gamma$ is guaranteed to be non-negative on the set $K$. For the sake of simplicity, equality constraints are ignored. Note that, without loss of generality, one can formulate any equality constraint $h(\bm{x}) = 0$ as a pair of two inequality constraints: $h(\bm{x})\geq 0, -h(x)\geq 0$.

The sum of squares (SOS) method is a sufficient certificate that replaces the non-negativity constraint in~\eqref{eq:hypograph} to make the problem computationally tractable. A polynomial written as a sum of squares is, by definition, non-negative over all $\bm{x}\in \mathbb{R}^d$. This certificate can be represented with a positive semi-definite cone. With the vector $z$ being all monomials of degree less than $\lceil n/2 \rceil$, we can write any polynomial of at most degree $n$ as
\begin{align*}
    f(\bm{x}) = z^\top Q z\ .
\end{align*}
Then, the polynomial $f(\bm{x})$ is sum of squares if and only if $Q\succeq 0$. This non-negativity certificate is used to construct a convergent hierarchy of lower bounds~\cite{lasserre2001, parrilo2003}. Assume $n$ denotes the largest degree among all polynomials involved in \eqref{eq:hypograph}. By replacing the non-negativity by a sufficient and tractable SOS certificate utilizing weak duality, we approximate~\eqref{eq:hypograph} as:
\begin{align}
\begin{split}
f^*\geq f_0^* = \max_{\gamma, \bm{\lambda}\geq 0} \ & \ \gamma,\\
\mathrm{s.t.} \ & \ f(\bm{x}) -\gamma - {\sum_i} \lambda_i g_i(\bm{x}) \in \mathrm{SOS}(n),
\end{split} \label{eq:dual}
\end{align}
where $\mathrm{SOS}(n)$ is the set of sum of squares polynomials up to degree $n$, defined as
\begin{align*}
    \mathrm{SOS}(n) = \{f(\bm{x}) \, &| \, f(\bm{x}) = {\textstyle\sum_i} f_i^2(\bm{x}) : f_i(\bm{x}) \in \mathbb{R}[\bm{x}] \\
    & \quad\ \text{and \  deg}(f_i(\bm{x}))\leq \lceil n/2 \rceil, \ \forall i \}.
\end{align*}
The subscript of $f_0^*$ in \eqref{eq:dual} denotes the level of the hierarchy. 

It is shown in~\cite{lasserre2001} that the optimality gap is reduced by introducing Lagrange multipliers which are SOS polynomials $\lambda(\bm{x})$ of a fixed degree:
\begin{align*}
\begin{split}
f_{p-1}^*\leq f_p^*& = \max_{\gamma} \  \ \gamma,\\
\mathrm{s.t.} \  \ f(&\bm{x}) -\gamma - \sum_i \lambda_i(\bm{x}) g_i(\bm{x})  \in \mathrm{SOS}(n+p), \\
 \ \lambda_i(&\bm{x}) \in \mathrm{SOS}(n+p-n^{(i)}),\quad \forall  i,
\end{split} 
\end{align*}
where $n^{(i)}$ denotes the degree of the polynomial constraint $g_i(\bm{x})$. Under some easily attained conditions, convergence of the hierarchy is guaranteed and whenever the duality gap is zero the solution can be recovered. The dual perspective of this hierarchy is referred to as the moment relaxation hierarchy~\cite{lasserre2001,schmuedgen1991}.

Notice that the problem size grows exponentially with both the number of variables $d$ and the degree $n$. Due to this phenomenon, the SOS method struggles when solving problems with a large number of variables and/or problems involving high degree polynomials. An alternative way of certifying non-negativity of polynomials is to instead utilize the relative entropy cone~\cite{chandrasekaran2016} (and the so called AM/GM polynomials). This hierarchy relies on the arithmetic mean geometric mean inequality and it will be referred to as the SAG hierarchy. As is the case for SOS, this approach also produces a convergent hierarchy of lower bounds. However, the certificate scales with the square of the number of unique monomials and can thus potentially produce better results than the SOS method for problems which are high degree and dimension but sparse. As a remark, both certificates/hierarchies can also be combined.

\section{Pulse Width Modulation Framework}\label{sec:pwm}

In Pulse Width Modulation (PWM) a reference voltage $v: [0,2\pi]\rightarrow \mathbb{R}$ is approximated as a switching signal $u: [0,2\pi]\rightarrow \{-1,0,1\}$. The number of possible switch positions is defined by the converter structure, e.g., in this case $u$ is defined for a three-level converter with three switching positions. The instances where we switch from one level to another are such that the output voltage follows some given reference $v$ as accurately as possible. The magnitude of the modulating signal is commonly referred to as the modulation index $m\in [0,\frac{4}{\pi}]$.

\subsection{Optimized Pulse Patterns}
Optimized Pulse Patters (OPPs) present an optimal approach to minimizing current harmonics. In this case a cost function is formulated, typically related to the current distortions, and then minimized subject to constraints~\cite{boller2011}.

The objective is usually chosen to be related to the current total demand distortion (TDD) in the stator, which for a three-phase, three-level converter connected to an induction motor is given by
\begin{align}
\begin{split}\label{eq:cost_complicated}
I_{\text{TDD}} =&\frac{\sqrt{2}V_{\text{DC}}}{\pi I_{\text{s,nom}}\omega_1 X_\sigma}  \\ &\times\sqrt{\sum_{n=5,7,11,\dots}\left(\frac{1}{n^2}\sum_{i=1}^{d}\Delta u_i \cos(n\alpha_i)\right)^2}\ ,
\end{split}
\end{align}
where $I_{\text{s,nom}}$ is the nominal root mean square stator current. $V_{\text{DC}}$ is the magnitude of the DC voltage, $\omega_1$ is the fundamental frequency and $X_\sigma$ is the leakage reactance of the motor driven by the converter. The magnitude of each harmonic is a function of the $d$ switching instances $\{\alpha_i\}_{i=1}^d$ in one quarter period. The fundamental component for $n=1$ is excluded. At the instance $\alpha_i$, the voltage is switched from $u_{i-1}$ to $u_i$, where $u_i\in\{-1,0,1\}$. Usually quarter- and half-wave symmetries are imposed and $u_0 = 0$ is assumed. Let us define $\Delta u_i = u_i - u_{i-1}$ for $i=1,2,\dots,d$ to be the switching transition. For simplicity, in this work, we predetermine $\Delta u_i = (-1)^{i+1}$ for $i = 1,2,\dots,d$, since we enforce that the maximum difference between subsequent transitions is always $1$.  

All the terms in~\eqref{eq:cost_complicated} except those inside the square root are constants and dependent on the parameters of the induction motor connected to the inverter; thus, the objective function to be minimized is
\begin{align*}
J(\bm{\alpha}) = \sum_{n = 5,7,11,\dots}\left(\frac{1}{n^2}\sum_{i=1}^{d}\Delta u_i \cos(n\alpha_i)\right)^2, 
\end{align*}
where $\bm{\alpha} = [\alpha_1,\alpha_2,\dots,\alpha_d]^\top$ is the vector of switching angles. 

The constraints of this problem are two-fold. First, the amplitude of the fundamental component ($n=1$) needs to match the  modulation index $m$, that is,
\begin{align*}
\hat{u}_1 = \frac{4}{\pi}\sum_{i=1}^{d}\Delta u_i \cos(\alpha_i) = m .
\end{align*}
Second, the order of the switching transitions must adhere to:
\begin{align*}
0\leq \alpha_1 \leq \alpha_2 \leq \dots \leq \alpha_d \leq \frac{\pi}{2} .
\end{align*}
Thus, the final optimization problem for a three-phase converter is given by: 
\begin{align}
\begin{split}
\min_{\bm{\alpha}} \ & \ \sum_{n=5,7,11,\dots}\left(\frac{1}{n^2}\sum_{i=1}^{d} \Delta u_i \cos(n\alpha_i) \right)^2 \\
\mathrm{s.t.} \ & \ \ \frac{4}{\pi}\sum_{i=1}^{d}\Delta u_i \cos(\alpha_i) = m, \\
& \ \ 0\leq\alpha_1\leq \alpha_2 \leq \dots\leq\alpha_d \leq \frac{\pi}{2}.
\end{split} \label{eq:opt_prob_orig}\
\end{align}
However, this yields a non-convex optimization problem. For tractability issues, the sum in the objective function is typically truncated to consider current harmonics up to a specified order, for instance, 300. Then, polynomial approximation of the truncated sum is performed, as shown in in~\cite{thesis_lukas}. Henceforth, we refer as objective function $J(\bm{\alpha})$ the truncated and approximated polynomial objective function. An example of a computed OPP for pulse number $d=3$ and modulation index $m=0.9$ is shown in Figure~\ref{fig:oppex}.
\begin{figure}
	\centering		
    \includegraphics[width=0.48\textwidth]{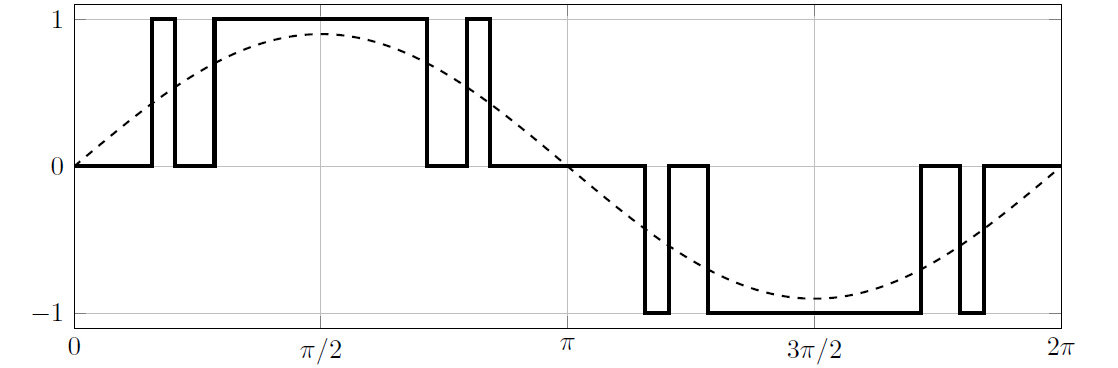}
	\caption[OPP example for $d=3$ for $m=0.9$]{OPP for $d = 3$ switching angles per quarter period and a modulation index of $m=0.9$ (solid line) and the modulating signal (dashed line).}
	\label{fig:oppex}
\end{figure}

\section{A Convex Relaxation Hierarchy for the OPP Problem}\label{sec:main}

Our goal is to solve \looseness=-1 
\begin{align}
\begin{split}
\bm{\alpha}^* = \arg \min_{\bm{\alpha}} \ & \ J(\bm{\alpha})\\
\mathrm{s.t.} \ & \ \frac{4}{\pi}\sum_{i=1}^d \Delta u_i \cos(2\pi \alpha_i) = m,\\
& \ 0\leq \alpha_1 \leq \alpha_2 \leq \dots \leq \alpha_d \leq \frac{1}{4}.
\end{split} \label{eq:opt_prob_poly}
\end{align}
Let the optimal cost be denoted as $J^* = J(\bm{\alpha}^*)$. Although the objective function is a polynomial, the equality constraint involves a trigonometric function. In the following, we use a Taylor-series approximation and derive its estimation error.

\subsection{Equality constraint approximation and its error}

To enable the use of polynomial methods for solving the OPP problem, the equality constraint is approximated using its Taylor-series expansion truncated after the $k$-th order, which we will denote as $\Gamma_{k}(\alpha)$, with an expansion point~$\alpha_0$:
\begin{align*}
\cos(2\pi \alpha) \approx \Gamma_k(\alpha) = \sum_{n=0}^k \frac{\dv[n]{\alpha}\cos(2\pi \alpha)|_{\alpha=\alpha_0}}{n!}(\alpha-\alpha_0)^n .
\end{align*} 

Let $R_k(\alpha)$ be the absolute error of a Taylor-series $\Gamma_k(\cdot)$ evaluated at $\alpha$. For any interval $(\alpha_0 -r, \alpha_0+r)$ with $r>0$, the inequality 
\begin{align*}
\begin{split}
R_k(\alpha) \leq (2\pi)^{k+1}\frac{|\alpha-\alpha_0|^{k+1}}{(k+1)!} \leq (2\pi)^{k+1}\frac{r^{k+1}}{(k+1)!}, \\ \quad \forall \alpha \in [\alpha_0-r,\alpha_0+r] ,
\end{split}
\end{align*}
holds, since 
\begin{align*}
    \dv[k+1]{\alpha}\cos(2\pi \alpha) \leq (2\pi)^{k+1}.
\end{align*} 
Let $\Delta m_k$ denote the error in the modulation index equality constraint in problem~\eqref{eq:opt_prob_poly}. It can be defined as
\begin{align}
\begin{split}
\Delta m_k = \frac{4d}{\pi}(2\pi)^{k+1}\frac{r^{k+1}}{(k+1)!} \geq\frac{4d}{\pi}R_k(\alpha)\label{eq:remainder},\\ \, \forall \alpha \in [0,\dfrac{1}{4}],
\end{split}
\end{align}
where $r=1/8$ and $\alpha_0=1/8$ such that $[\alpha_0-r,\alpha_0+r] = [0,1/4]$. In this context, the relaxation of the OPP problem~\eqref{eq:opt_prob_poly} is given by
\begin{align}
\begin{split}
J_k^* = \min_{\bm{\alpha}} \ & \ J(\bm{\alpha})\\
\mathrm{s.t.} \ & \ \frac{4}{\pi}\sum_{i=1}^d \Delta u_i \Gamma_k(\alpha_i)  \leq m + \Delta m_k,\\
& \ \frac{4}{\pi}\sum_{i=1}^d \Delta u_i \Gamma_k(\alpha_i)  \geq m - \Delta m_k,\\
& \ 0\leq \alpha_1 \leq \alpha_2 \leq \dots \leq \alpha_d \leq \frac{1}{4}.
\end{split} \label{eq:opt_prob_poly_approx}
\end{align}
Clearly, any feasible solution to \eqref{eq:opt_prob_poly} is also feasible to~\eqref{eq:opt_prob_poly_approx}, so we have $J_k^*\leq J^*$. We call the polynomial optimization \eqref{eq:opt_prob_poly_approx} the $k$-th level of the Taylor-series approximation hierarchy.
\subsection{Convergence of the hierarchy}
We now show that the relaxation of the original problem (combined with convex relaxation hierarchies for polynomial optimization problems) converges to the optimal solution of the original problem as the level of the hierarchy increases. We first prove two propositions that will facilitate our discussions.

\begin{proposition}
$\Delta m_k$ is strictly decreasing in $k$ when $r=1/8$.
\label{lem:mon}
\end{proposition}
\begin{proof} 
Monotonicity is established when considering two subsequent error terms in \eqref{eq:remainder}:
\begin{align*}
\Delta m_k  - \Delta m_{k-1} &=\frac{4d}{\pi}\left((2\pi)^{k+1} \frac{r^{k+1}}{(k+1)!} - (2\pi)^k \frac{r^k}{k!}\right)\\
&=\frac{4d}{\pi}(2\pi)^k\left(\frac{r^k (2\pi r - k - 1)}{(k+1)!} \right).
\end{align*}
Due to $d>0$, $r\geq 0$, and $k\geq 0$ for strict monotonicity, it must hold that
\begin{align*}
    2\pi r - k - 1 < 0,
\end{align*}
which is true for $r=1/8$. 
\end{proof}
\begin{proposition}
As $k\to \infty$, the problem~\eqref{eq:opt_prob_poly_approx} converges to the non-approximated problem~\eqref{eq:opt_prob_poly}.
\label{lem:conv}
\end{proposition}
\begin{proof}
One can immediately see that for an increased Taylor-series level the remainder term and thus also the error in the modulation index $\Delta m$ goes to zero:
\begin{align*}
\lim_{k\to \infty}\Delta m_k = \lim_{k\to \infty}  \frac{4d}{\pi}(2\pi)^{k+1}\frac{r^{k+1}}{(k+1)!}= 0\ .
\end{align*}
Thus, the problem does converge to its optimum, since any feasible solution to \eqref{eq:opt_prob_poly} is also feasible to~\eqref{eq:opt_prob_poly_approx}. 
\end{proof}

In general, a hierarchy is said to be complete, if two conditions are satisfied: \textit{(i)} the lower bounds the hierarchy produces need to be non-decreasing, and \textit{(ii)} it converges to the global optimum $J^*$. Let the lower bound obtained for the $p$-th level of the hierarchy be denoted by $J_p$, then we can write these two requirements as
\begin{enumerate}
    \item[\textit{(i)}] $\lim_{p\to \infty} J_p = J^*$,
    \item[\textit{(ii)}] $J_{\tilde{p}} \leq J_p \leq J^*, \quad \forall p \leq \tilde{p}$.
\end{enumerate}

Both of these requirements have been shown to hold for either the SAG and the SOS hierarchy under some regularity\footnote{For instance, in case of SOS, these so-called Archimedeanity conditions hold if we have some ball constraints over the feasible region~\cite{putinar1993}. These can easily be included in our problem since we are already working with box constraints.} conditions~\cite{chandrasekaran2016,lasserre2001, laurent2009}. The main interest now lies in whether the same holds for the level of the Taylor-series approximation. As a remark, in the SAG hierarchy, there are two parameters to determine the level of the hierarchy, that is, $p\in\R^2$~\cite{chandrasekaran2016, murray2018}.

Next, we study how the lower bound changes with level $k$. In Proposition~\ref{lem:conv}, we showed that the lower bound will (eventually) converge to the optimal one, achieving the property \textit{(i)} above. 

For monotonicity of the hierarchy,  we need to show that $J_k^*$ is nondecreasing in $k$. Even though $\Delta m_k$ is strictly decreasing in $k$ (as we showed in Proposition~\ref{lem:mon}), this result alone is not enough to conclude this property. We need that any solution for an approximation of level $k+1$ has to be a feasible solution to the level $k$. This monotonicity property holds for both the SOS and the SAG hierarchy in the level of the hierarchy $p$. In case of the Taylor-series-based approximation, we let $F_k(\bm \alpha) = \frac{4}{\pi}\sum_{i=1}^d \Delta u_i \Gamma_k(\alpha_i)-m$. Then, for any $\bm \alpha$ it should hold
\begin{align*}
    \forall \bm \alpha :  -\Delta m_{k+1}\le F_{k+1}(\bm \alpha)&\leq \Delta m_{k+1} \\
    \Rightarrow \quad -\Delta m_{k}\le F_k(\bm \alpha)& \leq \Delta m_k.
\end{align*}
For instance, for the right hand side, this is implied by
\begin{align*}
    \max_{\bm \alpha}[F_k(\bm \alpha)-F_{k+1}(\bm \alpha)]\leq \Delta m_k - \Delta m_{k+1}.
\end{align*}
To the best of our knowledge proving any of these two statements is not possible for the Taylor-series-based approximation. 

However, if the Taylor-series level $k$ is kept constant and the hierarchy level $p$ increases, the lower bound will always improve monotonically as compared to the previous level as
\begin{align*}
J_{p,k}\leq J_{\tilde{p},k} \leq J^*, \quad \forall p\leq \tilde{p},\ \forall k . \vspace{.1cm}
\end{align*}
This is an inherent property of the SOS/SAG hierarchy and is shown in~\cite{lasserre2001} and \cite{chandrasekaran2016}.
Considering the case where the approximation level $k$ goes to infinity and $p$ is large enough such that the SOS/SAG hierarchy converges, we can conjecture that
\begin{align*}
\exists p: J_{p,k{\rightarrow\infty}} =J^* . 
\end{align*}
Hence, the proposed hierarchy is convergent and eventually (as $k\to \infty$) attains the globally optimal solution.

\vspace{.1cm}

\section{Numerical Results}
In this section, we show that the previously discussed approximations can be used to solve the OPP problem to global optimality in an efficient manner. To this end, the OPP problems are also solved using a local optimization solver without any polynomial approximation to the equality constraint. The SOS/SAG hierarchy level $p$ and the Taylor-series approximation level $k$ will be subject to changes, so is the number of optimization variables~$d$. 

\vspace{.1cm}

\subsection{Solution of the OPP problem using polynomial optimization methods}
As a local optimization solver, \texttt{fmincon()} is used in MATLAB to solve the non-approximated problem~\eqref{eq:opt_prob_poly}. This problem is solved with 100 random initial conditions using gradient descent methods and the minimum was designated to be the global minimum. We first solve it for modulation index of $m=0.55$. The objective of problem~\eqref{eq:opt_prob_poly_approx} is a degree-three polynomial, see~\cite{thesis_lukas} for details. 
The SOS hierarchy is implemented via the YALMIP toolbox for MATLAB~\cite{lofberg2004, lofberg2009} and using the MOSEK solver~\cite{mosek}. The SAG hierarchy is implemented via the SAGEopt package for Python~\cite{murray2018, murray2019}, also using the MOSEK solver.\footnote{Note that in the SAG hierarchy and thus in the SAGEopt package two parameters determine the level of the hierarchy. For the sake of simplicity, the lowest level of the hierarchy is denoted as $p=1$, representing a set of input parameters of \texttt{p=0}, \texttt{q=1} (and \texttt{ell=0}) in the toolbox~\cite{murray2018,murray2019}.} Table~\ref{tab:num_exp_bern_error_bounds} illustrates the obtained results for a fixed truncation level of the Taylor-series approximation: $k=5$. The value $J(\bm{\alpha}^*)$ is the lower bound obtained for the chosen Taylor-series approximation level ($k=5$). Equivalently, this is the objective function $J$ evaluated at the optimizer $\bm{\alpha^*}$ returned after the convergence of the SOS/SAG hierarchies. The deviation of the resulting modulation index from the one imposed by the equality constraint is marked with $m_e$. It becomes apparent that the lower number of switching instances $d$ in the first quarter is, the lower the required Taylor-series approximation level is. The hierarchy converges in all cases for the reported hierarchy level. 
\begin{table*}\small
	\centering	
	\caption{Solver comparison for $k=5$ and $m=0.55$}
	\label{tab:num_exp_bern_error_bounds}
	\begin{tabular}{r|r|r|r|r|r|r}
		$d$ & Method & $J(\bm{\alpha}^*)$& $\bm{\alpha}^*$ & convergence level $p$& $m_\text{e}$ &time (s)  \\\hline
		1 & Local optimization solver & 0.1617 & [0.1789] & - &0.5500 &0.87 \\
		& YALMIP SOS toolbox & 0.1615&[0.1790]&4 & 0.5496 & 0.55 \\
		& SAGEopt toolbox& 0.1615&[0.1790]&1 & 0.5496 & 0.09\\\hline
		2 & Local optimization solver & 0.1554 & [0.1380; 0.2155] & - & 0.5500 &1.69 \\
		&  YALMIP SOS toolbox & 0.1549&[0.1380; 0.2154] & 4  &0.5492&0.49\\
		&SAGEopt toolbox & 0.1549 &[0.1379; 0.2154]&1 &0.5492&0.14  \\\hline
		3 & Local optimization solver & 0.1534 & [0.1151; 0.1639; 0.2184] & - & 0.5500 & 1.84 \\
		&  YALMIP SOS toolbox& 0.1527&[0.1155; 0.1645; 0.2188]&4 & 0.5487 &0.44\\
		& SAGEopt toolbox & 0.1527&[0.1154; 0.1644; 0.2188]& 1 & 0.5487 &0.16\\\hline		
		6 & Local optimization solver & 0.1519 & \makecell{[0.0817; 0.1028; 0.1425; \\0.1713; 0.2020; 0.2359]} & - & 0.5500&3.17 \\
		&  YALMIP SOS toolbox & 0.1506&\makecell{[0.0815; 0.1024; 0.1420;\\ 0.1704; 0.2009; 0.2351]}&4 & 0.5476& 19.5\\
		& SAGEopt toolbox &0.1506 &\makecell{[0.0814; 0.1023; 0.1420;\\ 0.1704; 0.2009; 0.2351]}&1 & 0.5476& 1.1\\\hline
		
	\end{tabular}
\end{table*}
\subsection{Effect of approximation level on the optimal solution}
A similar convergent behaviour is also apparent in Figure~\ref{fig:comp_TO_pulse_no}, where we report the deviation of the obtained lower bound compared to the actual minimum obtained using \texttt{fmincon()} when solving the Taylor-series approximation for different pulse numbers. Again a modulation index of $m=0.55$ is used. In our experiments, up to $d=20$, an approximation level of $k = 6$ was sufficient for reasonably close results to the actual optimum. 
\begin{remark}
Note that the convergence of the hierarchy was shown for $k\rightarrow\infty$. Our numerical experiments have also illustrated that monotonicity generally holds for the level of the Taylor-series approximation. Hence, increasing the approximation level is expected to yield an improved lower bound. 
\end{remark}
\vspace{.1cm}

For $d=50$ we need to go up to a Taylor-series approximation level of $k=8$ to be reasonably close, which can be observed in Table~\ref{tab:num_exp_high_pulse}. 
\begin{table}\small
	\centering
	\caption{$d=50$, $m=0.55$, effect of approximation level $k$ on optimum}
    \label{tab:num_exp_high_pulse}
	\begin{tabular}{r|r|r|r|r|r}
		d & Method & $k$ &$J(\bm{\alpha}^*)$& $m_\text{e}$ &time (s)  \\\hline
		50 & Local solver& - & 0.1513 & 0.5500&258 \\
		& SAGEopt toolbox& 5 &0.1405  &  0.5292 & 388\\
		& SAGEopt toolbox & 8 & 0.1513&0.5500 & 508\\\hline
	\end{tabular}
\end{table}
\begin{figure}
	\centering	
    \includegraphics[width=0.44\textwidth]{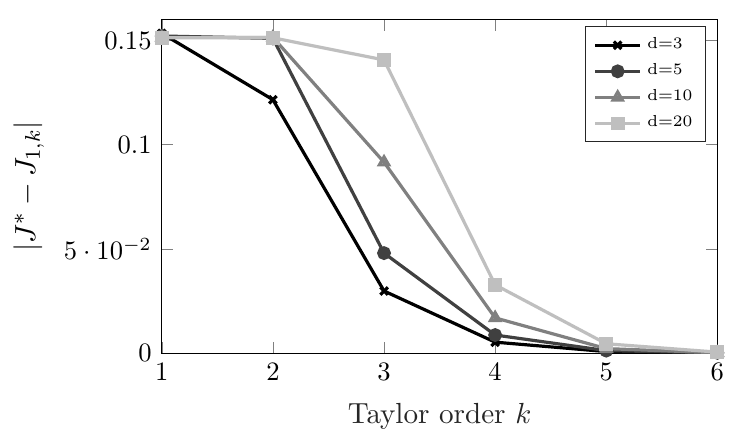}
    \vspace{.1cm}
	\caption{Effect of the approximation level $k$ on the optimal cost. $|J^*-J_{1,k}|$ is the absolute deviation of the optimal cost of the problem with approximation level $k$ from the actual optimal cost $J^*$, obtained with \texttt{fmincon()}. The approximated problem did converge for a hierarchy level of $p=1$.}
	\label{fig:comp_TO_pulse_no}
	\vspace{.1cm}
\end{figure}

Observe that Taylor-series levels are quite low, but a global optimum can be found within reasonable time. Going to larger pulse numbers or SOS/SAG levels proves to be computationally demanding due to the scalability of these polynomial optimization methods in both the degree of the polynomial and the number of optimization variables. 

While there are many efforts in the literature to improve the performance and scalability of the methods, most of these approaches try to exploit sparsity or symmetry in the objective, which are not applicable to the problem at hand. The objective function of problem~\eqref{eq:opt_prob_poly} contains all the unique monomials, which means that the size of the problem for higher levels of the hierarchy is still too large to be feasible for currently available SDP or relative entropy (exponential) cone solvers.

\vspace{.1cm}

\section{Conclusion}\label{sec:conc}

We showed that the OPP problem can be solved efficiently via methods for polynomial optimization using a Taylor-series approximation to the equality constraint. While the solution using local methods is already well-established, the use of polynomial methods comes with the advantage of producing a global solution. We formed a convex relaxation hierarchy based on a Taylor-series approximation and well-establish SOS/SAG hierarchies of polynomial optimization problems. We showed that this hierarchy is convergent, and the approximation error of the equality constraint is monotonically decreasing. We  showed  how these methods  provide  us  with  sufficiently  precise  results for  a  certain  level  of  the Taylor-series approximation  in  a  reasonable computation time.

Our future work will focus on larger pulse numbers and polynomial degrees which still prove to be computationally demanding. 

\vspace{.1cm}

\section*{Acknowledgements}
The authors would like to thank Tobias Geyer at ABB Corporate Research Center in Baden, Switzerland for sharing his expertise on OPPs and giving his thoughts on the progress of the research. The authors would also like to thank Gernot Riedel at ABB Corporate Research Center in Baden, Switzerland for making the completion of this work possible in his research group.

\vspace{.1cm}

\bibliographystyle{IEEEtran}
\bibliography{IEEEabrv,library}

\end{document}